\documentclass{amsart}
\usepackage{amsmath,amsthm,latexsym,amsfonts,amssymb,mathrsfs,array,manfnt}
\usepackage[all]{xy}
\usepackage{paralist}
\usepackage{framed}
\usepackage{cancel}
\usepackage{enumitem}

\usepackage{units}
\usepackage{gensymb}
\usepackage{setspace}
\usepackage{amscd}
\usepackage{tikz}
\usetikzlibrary{matrix}

\newcommand{\abs}[1]{\left\lvert#1\right\rvert}

\newcommand{\Z}{\ensuremath{\mathbb{Z}}}
\newcommand{\R}{\ensuremath{\mathbb{R}}}
\newcommand{\C}{\ensuremath{\mathbb{C}}}
\newcommand{\Q}{\ensuremath{\mathbb{Q}}}

\renewcommand{\P}{\ensuremath{\mathbb{P}}}

\renewcommand{\bar}[1]{\overline{#1}}

\DeclareMathOperator{\num}{N}

\DeclareMathOperator{\Id}{Id}

\DeclareMathOperator{\GL}{GL}
\DeclareMathOperator{\SL}{SL}

\usepackage{color}

\newcommand{\KK}{\mathbb{K}}

\newcommand{\PP}{\mathbb{P}}
\newcommand{\QQ}{\mathbb{Q}}

\newcommand{\ZZ}{\mathbb{Z}}

\theoremstyle{plain}
\newtheorem{theorem}{Theorem}
\numberwithin{theorem}{section}

\newtheorem{thm}[theorem]{Theorem}

\newtheorem{lem}[theorem]{Lemma}
\newtheorem{question}[theorem]{Question}

\theoremstyle{definition}
\newtheorem{Definition/Theorem}[theorem]{Definition/Theorem}
\newtheorem{Definition/Proposition}[theorem]{Definition/Proposition}

\newtheorem{Def}[theorem]{Definition}

\newtheorem{ex}[theorem]{Example}

\newtheorem{Corollary/Definition}[theorem]{Corollary/Definition}

\theoremstyle{remark}

\newtheorem{rem}[theorem]{Remark}

\newcommand{\Com}{TCom}
\newcommand{\T}{\mathbb{T}}

\renewcommand{\setminus}{\smallsetminus}

\DeclareMathOperator{\PD}{PD}

\usetikzlibrary{arrows}
\usetikzlibrary{calc}

\newcommand{\ratt}{
\tikz[minimum height=2ex]
  \path[dashed,->]
   node (a)            {}
   node (b) at (1em,0) {}
  ($(a.center)+(0,.0)$) edge ($(b.center)+(0,.0)$)
  ($(a.center)+(0,-.14)$) edge ($(b.center)+(0,-.14)$);
}

\usepackage{hyperref}

\begin{document}

\title{Dynamical invariants of monomial correspondences}
\author{Nguyen-Bac Dang, Rohini Ramadas}
\email{nguyen-bac.dang@stonybrook.edu, rohini\_ramadas@brown.edu}
\thanks{This work was partially supported by NSF grant 1703308.}

\begin{abstract}
We focus on various dynamical invariants associated to monomial correspondences on toric varieties, using algebraic and arithmetic geometry. We find a formula for their dynamical degrees, relate the exponential growth of the degree sequences with a strict log-concavity condition on the dynamical degrees and compute the asymptotic rate of the growth of heights of points of such correspondences.
   \end{abstract}
\maketitle


\section*{Introduction}\label{sec:Intro}

Let $\KK$ be an algebraically closed field of characteristic zero, $X$ be a smooth $n$-dimensional quasiprojective variety over $\KK$. A \textit{rational correspondence} $g$ on $X$ is a quasiprojective variety $\Gamma_f$ together with two maps $\pi_1$ and $\pi_2$ to $X$, such that both maps are dominant when restricted to every irreducible component of $\Gamma_f$. Fix a normal compactification $\bar X$ of $X$, a compactification $\bar\Gamma_f'$ of $\Gamma_f$ extending the maps $\pi_1$ and $\pi_2$ to $\bar X$, and a desingularization $\bar \Gamma_f$ of  $\bar \Gamma_f'$.

Rational correspondences can be iterated (see Section \ref{sec:RationalCorrespondences}); we denote by $f^p$ the $p$-th iterate of the correspondence $f$. There are invariants of various nature associated to the dynamical system induced by a given correspondence. They can be transcendental or algebraic, or arithmetic if $\KK$ is $\bar\QQ$. 

Let us present a first invariant defined purely using algebraic geometry.
Fix an ample divisor $D$ on $\bar X$ and an integer $k\leqslant n$, the $k$-degree of the correspondence $g$ with respect to $\bar X $ and $D$ is the intersection number:
\begin{equation*}
\deg_k(f) :=( \pi_1^* D^{n-k} \cdot \pi_2^* D^k )
\end{equation*}
on $\bar \Gamma_f$.
The asymptotic behavior of the sequence $(\deg_k(f^p))_p$ roughly measures the algebraic complexity of the iterates of $f$. These degree sequences have been studied most in the case when $\Gamma_f$ is the graph of a rational map $f$.
 The growth of the sequence of degrees is an essential tool when one studies the group of birational transformations  (see \cite{gizatullin_80}, \cite{DillerFavre2001}, \cite{cantat_bir_surfaces}, \cite{blanc_deserti}, \cite{BlancCantat2016} for surfaces,  \cite{dinh_sibony_2004_groupes} for the study of commutative automorphism groups in dimension $\geq 3$ and \cite{cantat_zeghib_2012}, \cite{zhang_2014} for some characterizations of positive entropy automorphisms in higher dimension).

%
When $\KK$ is the field of complex numbers, Dinh-Sibony \cite{DinhSibony2005, DinhSibony_correspondence} proved that the sequence of degrees is submultiplicative and this result was recently extended to a field of arbitrary characteristic \cite{Truong2016,dang2017degrees}. 
As a result, Fekete's lemma implies that the asymptotic ratio of the sequence $\deg_k(g^p)$, called the $k$-th dynamical degree of $g$ and denoted $\lambda_k(g)$ is well-defined and is equal to the limit
\begin{equation*}
\lambda_k(g) :=\lim_{p\rightarrow +\infty} \deg_k(g^p)^{1/p}.
\end{equation*}
These numbers were first defined for rational maps over the complex projective space by Russakovski-Shiffman \cite{russakovskii_shiffman} and are in general birational invariants \cite{Truong2016}.  
Dynamical degrees are also a key ingredient in the construction of ergodic invariant measures (\cite{bedford_smillie_3}, \cite{sibony_1999},\cite{dinh_sibony_dyn_reg_Pk},  \cite{guedj_2005}, \cite{de_thelin_vigny_2010}, \cite{dinh_sibony_2010}, \cite{diller_dujardin_guedj_2011_II}).

The quantity analogous to the dynamical degree in the analytic setting is the topological entropy --- the exponential growth rate of the number of $(n,\epsilon)$ separated orbits that avoid the indeterminacy locus of a given correspondence. 
These two invariants are closely related as the dynamical degrees control the topological entropy (\cite{Gromov2003, DinhSibony2005, DinhSibony_correspondence}) and the equality between the topological entropy and the logarithm of its largest dynamical degree is achieved for holomorphic maps \cite{Yomdin1987}. 
%

Like the entropy, the dynamical degrees are also difficult to compute in general. While the $n$th dynamical degree $\lambda_n$ is the topological degree of $g$, the issue when $1 \leqslant k<n$ is that  $g$ has in general a non-empty indeterminacy locus and hence the pullback $(g^p)^*$ on the Neron Severi group need not be equal to $(g^*)^p$. Thus computing $\lambda_k$ involves computing infinitely many potentially unrelated pullback maps. 
Therefore, they have only been computed in low dimension or for maps which preserve  certain geometric constraints: they are known for regular morphisms, for birational maps of surfaces \cite{DillerFavre2001}, for endomorphisms of the affine plane \cite{FavreJonsson2011}, for monomial maps \cite{lin,favre_wulcan}, for birational maps of hyperk\"ahler varieties \cite{bianco} and for Hurwitz correspondences (a class of mappings and correspondences obtained from Teichm\"uller theory in the work of Koch \cite{Koch2013}) \cite{KochRoeder2015, Ramadas2015, Ramadas2016}. The properties of dynamical degrees were studied for particular classes of birational transformations of $\PP^3$ \cite{deserti_han,cerveau_deserti,deserti_leguil,dang_tame}. Due to this difficulty, there are many open questions. Until recently, it was not known whether every dynamical degree is an algebraic integer; however just this year Bell-Diller-Jonsson \cite{BellDillerJonsson2019} have found a map with a transcendental dynamical degree. 
 
 In contrast to the general situation, the dynamical degrees and the sequence of degrees of monomial maps are well understood \cite{bedford_kim_linear,lin,favre_wulcan, JonssonWulcan2011,dang_xiao}. Fix $n>0$, and denote by $\T$ the $n$-dimensional torus $(\KK^*)^n$. Let $M=(M_{ij})$ be an $n\times n$ integer matrix. Then $M$ induces a monomial self-map $\phi(M)$ of $\T$, sending $(x_1,\ldots x_n)$ to $(\prod_jx_{j}^{M_{1j}},\ldots,\prod_{j}x_j^{M_{nj}})$. If $M$ is non-singular then $\phi(M)$ is dominant, with topological degree $\det(M)$. The $k$-th dynamical degree of $\phi(M)$ is 
 \begin{equation}
|\rho_1| \cdot \ldots \cdot |\rho_k|, 
\end{equation}
where $\rho_1, \ldots ,\rho_k$ are the $k$ largest eigenvalues of $M$.
 
Now let $M$ and $N$ be two nonsingular integer matrices. Then we have two dominant maps, $\phi(M)$ and $\phi(N)$, both from $\T$ to $\T$. This induces a rational correspondence on $\T$; we call such a correspondence a \textit{monomial correspondence}. Here, we compute the dynamical degrees of monomial correspondences. We show:
\medskip
 
\noindent \textbf{Theorem A}. 
  For any two $n\times n$ integer matrices $M$ and $N$ with non-zero determinant, the $k$-dynamical degree of the correspondence $(\T, \phi(M),\phi(N))$ is equal to 
\begin{equation*}
|\det(M)| |\rho_1| \cdot \ldots \cdot |\rho_k|, 
\end{equation*}
where $\rho_1, \ldots ,\rho_k$ are the $k$ largest eigenvalues of the matrix $N \cdot M^{-1}$. 
\bigskip

We obtain further information on the growth of the degrees when the sequence $k \mapsto \lambda_k(f)$ is locally  strictly log-concave.
\medskip

\noindent \textbf{Theorem B}. 
 Fix two $n\times n$ integer matrices $M$ and $N$ with non-zero determinant and take $f$ the  monomial correspondence $f := (\T, \phi(M), \phi(N))$. Suppose that $\lambda_l^2(f) > \lambda_{l+1}(f) \lambda_{l-1}(f)$ for an integer $1\leqslant l \leqslant n$, then there exists a constant $C>0$ and an integer $r$ such that 
\begin{equation*}
\deg_l (f^p) = C \lambda_l(f)^p + O\left (p^r \left ( \dfrac{\lambda_{l-1}(f) \lambda_{l+1}(f)}{\lambda_l(f)}\right )^p \right )
\end{equation*}
\bigskip

We find an appropriate toric compactification $Y$ on which the pullback on the $2k$ cohomology is functorial (i.e $(f^p)^* = (f^*)^p $ on $H^{2k}(Y(\C))$ for an embedding of $\KK$ into $\C$). We say that $f$ is $k$-stable when this happens.
\medskip

\noindent \textbf{Theorem C}. 
Fix two $n\times n$ integer matrices $M$ and $N$ with non-zero determinant and take $f$ the  monomial correspondence $f := (\T, \phi(M), \phi(N))$.
Suppose that the eigenvalues of the matrix $N M^{-1}$ are all real, distinct and positive. Then there exists a projective toric variety whith at worst quotient singularities on which $f$ is algebraically $k$-stable.
\bigskip

Finally, we describe the monomial correspondences by computing these invariants using arithmetic tools.
Recall that given an ample divisor $D$ on $\PP^n$, one can associate a function $h_D : \PP^n(\Q) \to \mathbb{R}^+$ called a Weil height. 
These heights are essential  to understand rational points and integral points on algebraic varieties. More precisely, they are used to count the asymptotic growth of rational points \cite{schanuel}, to characterize torsion points on abelian varieties \cite{neron,tate} and to obtain equidistribution results in algebraic dynamics \cite{szpiro1997equirepartition,yuan,demarco2017variation} and to study stability properties of algebraic families of one dimensional maps \cite{masser_zannier,baker_demarco,demarco_bifurcation_heights,
demarco_wang_ye,favre_gauthier}.
In our setting, we fix an embedding from $\T$ into $\PP^n$ so that we get an injection from $\T(\bar \Q) $ in $\PP^n(\bar\Q)$. Any monomial map or correspondence is regular on the torus $\T$, as are all its iterates, so the indeterminacy locus on any compactification $X$ is supported on the boundary $X\setminus \T$. Thus if $x$ is a point in $\T(\bar \Q)$, its image by a monomial correspondence is a well-defined cycle of dimension $0$, i.e a finite number of $\T(\Q)$ points counted with multiplicities $\sum a_i [x_i]$ for $a_i \in \Z, x_i \in \T(\bar \Q)$. We define the height of the image as the sum $\sum a_i h_D(x_i)$.
Take $f := \left ( \T , \phi(M) , \phi(N) \right )$  a monomial correspondence with $M,N$ two matrices with integer entries and fix a $x$  a point in $\T(\bar \Q)$. We define the arithmetic degree of $x$ as the following asymptotic limit:
\begin{equation*}
\alpha_f(x) := \limsup_{p\rightarrow +\infty} h_D(f^p(x))^{1/p}. 
\end{equation*}
 
 In \cite{Kawaguchi_Silverman}, Kawaguchi and Silverman defined this quantity and conjectured that when $f$ is dominant rational map and when $x$ is a rational point whose orbit is Zariski dense, the quantity $\alpha_f(x)$ is equal to the first dynamical degree of $f$. 
\medskip

\noindent \textbf{Theorem D}.
Suppose that the field $\KK= \bar\Q$ is the field of algebraic numbers. 
Fix two $n\times n$ integer matrices $M$ and $N$ with non-zero determinant and take $f$ the  monomial correspondence $f := (\T, \phi(M), \phi(N))$. Then for any point $x \in \T(\bar \Q)$, the quantity $\alpha_f(x)$ is finite and belongs to the set $$\{1 , |\det(M)\rho_1|, |\det(M)\rho_2| , \ldots , |\det(M)\rho_n| \}$$ where each $\rho_i$ is an eigenvalue of $N \cdot M^{-1}$.
\bigskip

Let us explain how we can obtain these four results. Fix $M,N$ two $n \times n$ integer matrices with non-zero determinant.
Our approach relies on a very simple observation: mainly, by  post-composing by the monomial map associated to the matrix $\det(M)^p \Id$, we obtain the following diagram:
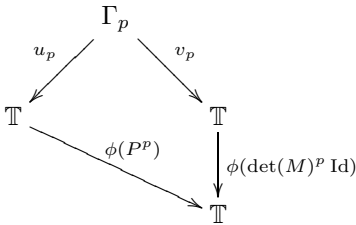
\begin{equation*}
\xymatrix{ & \Gamma_p \ar[rd]^{v_p} \ar[ld]_{u_p} & \\
 \T \ar[rrd]^{\phi(P^p)} & & \T \ar[d]^{\phi(\det(M)^p \Id)} \\
  & & \T}
\end{equation*}
 where $\Gamma_p$ is the correspondence associated to $g^p$, $\Com(M)$ is the transpose of the cofactor matrix of $M$, $P$ is the matrix $ \Com(M) \cdot N$,  and $\T$ is the torus $\KK^n$.
 This diagram allows us to transport the dynamical properties of the monomial map induced by $P$ with the dynamics of the monomial correspondence. 
 We finally conclude using Favre-Wulcan's \cite{favre_wulcan}, Lin-Wulcan's results \cite{lin_wulcan}, and Silverman's result \cite{silverman} to prove Theorem A, B, C and D respectively.
\medskip

To pursue the study of these particular correspondences, it is natural to ask whether we can relate the entropy of such correspondence with their dynamical degrees as in the rational setting (\cite{hasselblatt_propp}). Precisely, we formulate the following question.

\begin{question}  Take $M$ and $N$ two $n\times n$ integer matrices with non-zero determinant. Is the topological entropy of a monomial correspondence $f = (\T, \phi(M),\phi(N))$  equal to
\begin{equation*}
\max_{0\leqslant k \leqslant n} \lambda_k(f). 
\end{equation*}
\end{question}

Once this question is answered, then one would wish to understand the ergodic properties of our correspondences and construct interesting measures for this class of correspondences similar to the approach of \cite{guedj_2005} for rational maps and to \cite[Theorem 1.3]{dinh_sibony_distribution}, mainly:

\begin{question} When all the iterates of the monomial correspondence are irreducible, is there a measure of maximal entropy and can we compute its Hausdorff dimension?
\end{question}

\subsection*{Structure of the paper}

In \S 1 we review the background on correspondences, degrees and dynamical degrees. Following this, we prove successively in \S 2 Theorem A,B,C and D.

\subsection*{Acknowledgements}
The first author is grateful to Eric Bedford and Mikhail Lyubich for their support on this project, both authors thank Mattias Jonsson for his encouragements and the second author thanks Melody Chan for useful conversations.

\section{Background}\label{sec:Background}
\subsection{Rational
  correspondences}\label{sec:RationalCorrespondences} A
rational correspondence from $X$ to $Y$ is a multi-valued map to $Y$
defined on a dense open set of $X.$ When  $\KK$ is the field of complex numbers, they are also called meromorphic multi-valued maps.
\begin{Def}
  Let $X$ and $Y$ be irreducible quasiprojective varieties. A
  \emph{rational correspondence} $f=(\Gamma_f,\pi_X,\pi_Y):X\ratt
  Y$ is a diagram
  
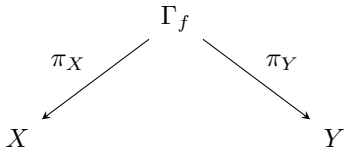
\begin{center}
    \begin{tikzpicture}
      \matrix(m)[matrix of math nodes,row sep=3em,column
      sep=4em,minimum width=2em] {
        &\Gamma_f&\\
        X&&Y\\}; \path[-stealth] (m-1-2) edge node [above left]
      {$\pi_X$} (m-2-1); \path[-stealth] (m-1-2) edge node
      [above right]
      {$\pi_Y$} (m-2-3);
    \end{tikzpicture}
  \end{center}
  where $\Gamma_f$ is a quasiprojective variety, not necessarily
  irreducible, and the restriction of $\pi_X$ to every irreducible
  component of $\Gamma_f$ is dominant and generically finite. We say that the variety $\Gamma_f$ is the graph of the correspondence $f$. 
\end{Def}
Over some dense open set in $X$, $\pi_X$ is an \'etale map (over $\C$, a covering map), and
$\pi_Y\circ\pi_X^{-1}$ defines a multi-valued map to $Y.$ We define the \emph{domain of definition} of $g$ to be the largest such open set of $X$, However,
considered as a multi-valued map from $X$ to $Y,$ it is possible
that $\pi_Y\circ\pi_X^{-1}$ has indeterminacy, since some fibers of
$\pi_X$ may be empty or positive-dimensional.

\begin{ex} Let $f:X\to Y$ be a rational map. Set $\Gamma_f$ to be the graph of $f$, i.e. the set $(x,f(x))\in X\times Y$, and set $\pi_X$ and $\pi_Y$ to be the natural projection maps from $\Gamma_f$ to $X$ and $Y$ respectively. Then $(\Gamma_f,\pi_X,\pi_Y):X\ratt
  Y$ is a rational correspondence, that happens to be generically single-valued. Thus any rational map can be thought of as a rational correspondence. 
\end{ex}

\begin{ex}
Let $\phi$ be an orientation-preserving branched covering from $S^2$ to itself, such that every critical point of $\phi$ has finite forward orbit. Thurston \cite{DouadyHubbard1993} described a pullback map induced by $\phi$ on the Teichm\"uller space of complex structures of $S^2$ marked at the post-critical locus of $\phi$; Koch \cite{Koch2013} showed that Thurston's pullback map descends to a correspondence on the moduli space $\mathcal{M}_{0,n}$ of configurations of $n$ points on $\P^1$. These correspondences are called \textit{Hurwitz correspondences} and have been studied in \cite{Ramadas2015, Ramadas2016}.
\end{ex}

\begin{ex} 
Recall from that the modular surface $X_0 :=\SL_2(\Z)//\mathbb{H}$ defined as the left quotient of the hyperbolic plane by $\SL_2(\Z)$ is isomorphic to $\C$.
We can thus view $X_0$ as an algebraic Riemann surface, whose compactification is the Riemann sphere. 
The Hecke operator of weight $n$ is a  morphism on the free abelian group generated by the rank two lattices of $\R^2$ such that for any lattice $\Lambda$ in $\R^2$, the image $T(n) (\Lambda)$ is defined as:
\begin{equation*}
T(n) [\Lambda] = \sum_{[\Lambda': \Lambda] = n} [\Lambda']. 
\end{equation*}  
Since $X_0$ is the quotient of the space of rank $2$ lattices of $\R^2$ up to a scaling by \cite[VII \S 2.2 Proposition 3]{serre_course_arithmetic}, the map $T(n)$ descends to a correspondence on $X_0$.
\end{ex}

\begin{Def}\label{Def:Composite}
  Suppose $f=(\Gamma_f,\pi_X,\pi_Y):X\ratt Y$ and
  $g=(\Gamma_{g},\pi_Y',\pi_Z'):Y\ratt Z$ are rational correspondences such
  that the image under $\pi_Y$ of every irreducible component of
  $\Gamma_f$ intersects the domain of definition of the multi-valued map
  $\pi_Z'\circ(\pi_Y')^{-1}.$ The \emph{composite}
  $g\circ f$ is a rational correspondence from $X$ to $Z$
  defined as follows.

  Pick dense open sets $U_X\subseteq X$ and $U_Y\subseteq Y$ such
  that $\pi_Y(\pi_X^{-1}(U_X))\subseteq U_Y,$ and
  $\pi_X|_{\pi_X^{-1}(U_X)}$ and $\pi_Y'|_{(\pi_Y')^{-1}(U_Y)}$ are
  both \'etale. Set
  $$g\circ
 f:=(\pi_X^{-1}(U_X)\thickspace\thickspace{_{\pi_Y}\times_{\pi_Y'}}\thickspace\thickspace(\pi_Y')^{-1}(U_Y), \pi_1, \pi_2),$$
   where $\pi_1$ and $\pi_2$ are the natural maps to $X$ and $Z$ respectively.
\end{Def}
This composite does depend on the choices of open sets $U_X$ and
$U_Y$, but is well-defined up to conjugation by a birational transformation. 

\subsection{Pullback on numerical groups by correspondences}

Fix $\bar X$ a normal projective compactification of $X$.
We now recall how these correspondences induce a natural action on the numerical groups of $k$-cycles. 

When $\KK$ is the field of complex numbers and $\bar X$ is smooth, then one can consider $\bar \Gamma_f$ the desingularization of a compactification of the correspondence $f=(\bar \Gamma_f,\pi_1,\pi_2):\bar X\ratt
  \bar X$. One can pull back smooth forms along $\pi_2$; this induces a linear map $\pi_2^* : H^{2k}(\bar X) \to H^{2k}(\bar\Gamma_f)$ on de Rham cohomology groups. Also, one can push forward homology classes along $\pi_1$, obtaining a map ${\pi_1}_* :H_{2n-2k}(\bar \Gamma_f) \to H_{2n-2k}(\bar X)$. Finally, since $\bar X$ and $\bar\Gamma_f$ are smooth oriented compact manifolds, Poincar\'e duality induces identifications $\PD_{\bar\Gamma_f}:H^{2k}(\bar \Gamma_f) \simeq H_{2n-2k}(\bar \Gamma_f)$ and $\PD_{\bar X}:H^{2k}(\bar X) \simeq H_{2n-2k}(\bar X)$. We may compose pullback, pushforward, and Poincar\'e duality maps to obtain a pullback along the correspondence $f$:
  
  $$f^*:= \PD_{\bar X}^{-1}\circ{\pi_1}_*\circ \PD_{\bar\Gamma_f}\circ\pi_2^*: H^{2k}(\bar X) \to H^{2k}(\bar X)$$  
  
In order to be able to work over an arbitrary algebraically closed field of characteristic zero, we replace the de Rham cohomology by abelian groups called numerical groups. We first introduce the general terminology on these groups. 

A $k$-cycle on $\bar X$ is a formal linear combination of subvarieties $\sum a_i [V_i]$ where $a_i$ are real numbers and where $V_i$ are subvarieties of $\bar X$. The group of $k$-cycles on $\bar X$ is denoted $Z_k(X)$. The rational equivalence classes of $k$-cycles form the group $A_k(\bar X)$.
In \cite[Chapter 3]{fulton_intersection}, Fulton introduces the Chern classes of a vector bundle $E$ on $\bar X$ of degree $k$ as operators from $A_l(\bar X) \to A_{l-k}(\bar X)$. Chern classes can be composed and a product of Chern classes $\alpha$ of degree $k$ is said to be numerically equivalent to zero if the intersection $(\alpha \cdot z)$ is equal to zero for any $k$-cycle $z$. The numerical group of codimension $k$ of $\bar X$, denoted $\num^k(\bar X)$ is the abelian group generated by product of Chern classes of degree $k$ modulo the numerical equivalence relation.
 Dually, we denote by $\num_k(\bar X)$ the quotient of the abelian group of $k$-cycles on $\bar X$  by the group generated by cycles $z\in Z_k(\bar X)$ satisfying $(\alpha \cdot z)=0$ for any product of Chern classes of degree $k$.
 When $\bar X$ is smooth or has at worst quotient singularities, then the two groups $\num^k(\bar X)$ and $\num_{n-k}(\bar X)$ are isomorphic and the isomorphism is realized by intersecting with the fundamental class $[\bar X]$.
 Classes in $\num^k(\bar X)$ can be pulled back and conversely classes in $\num_k(\bar X)$ can be pushed forward.

\begin{Def}
Suppose $f=(\Gamma_f,\pi_1,\pi_2):X\ratt X$ is a rational correspondence, and $\bar{X}$ is a normal projective variety birational to $X$. Take a desingularization $\bar{\Gamma_f}$ of some compactification of $\Gamma_f$ with the property that $\pi_1$ and $\pi_2$ are regular maps on $\bar{\Gamma_f}$.
The pullback map  on the numerical groups, denoted $f^\bullet: \num^k(\bar X) \to \num_{n-k}(\bar X)$ is given by the following intersection product:
 $$f^\bullet(\alpha) := {\pi_1}_*(\pi_2^*\alpha \cdot [\bar \Gamma_f]) \in \num_{n-k}(\bar X),$$
 where $\alpha$ is a class in $ \num^k(\bar X)$. 
\end{Def}

\subsection{Dynamical degrees}\label{sec:DynamicalDegrees}
\begin{Def}
Let $\dim X=n$, and let $f=(\Gamma_f,\pi_1,\pi_2):X\ratt X$ be a rational
  correspondence. Fix a smooth projective compactification $\bar{X}$ of $X$, a projective compactification and $D$ an ample divisor class on $\bar X$. Now pick a projective compactification $\bar{\Gamma}_f$ of $\Gamma_f$ such that both maps $\pi_1,\pi_2:\bar{\Gamma}_f\to\bar{X}$ are regular. Now, for $k=0,\ldots,n$, the $k$-degree of $f$ with respect to $D$ is the intersection number on $\bar{\Gamma}_f$: $$\deg_k(f,\bar X, D):=(\pi_1^*D^{n-k}\cdot \pi_2^*D^{k}).$$
  Note that while this intersection number very much does depend on the choices of $\bar{X}$ and $D$, by the projection formula, it is independent of the choice of compatible compactification $\bar{\Gamma}_f$.
  \end{Def}

\begin{Def}\label{Def:Dominant}
  Let $f=(\Gamma_f,\pi_1,\pi_2):X\ratt X$ be a rational
  correspondence such that the restriction of $\pi_2$ to every
  irreducible component of $\Gamma_f$ is dominant.  In this case we say
  $f$ is a \emph{dominant} rational self-correspondence. 
\end{Def}

\begin{Def}
  Let $\Gamma_f$ be as in Definition \ref{Def:Dominant}. Let $f^p=f\circ \ldots \circ f$ ($p$ times). Pick any normal projective variety  $\bar{X}$ birational to $X$ and fix $D$ an ample divisor class on $\bar X$. The \emph{$k$th
    dynamical degree} $\lambda_k(f)$ of $\Gamma_f$ is defined to be
  $$\lambda_k(f):=\lim_{p\to\infty}\left(\deg_k(f^p,\bar{X}, D)\right)^{1/p}.$$ 
\end{Def}
This limit exists and does not depend on the
  choice of the normal projective variety $\bar{X}$ birational to $X$ nor on the choice of the ample divisor $D$. Moreover, the sequence $k \mapsto \lambda_k(f)$ is log-concave and we shall refer to \cite{Truong2016} for the general properties of these quantities.  
\subsection{Monomial correspondences}\label{sec:monomial}

Fix $n>0$ and let $\T$ be the torus $(\KK^*)^n$.
Recall that any $(n\times n)$ matrix $M= (M_{ij})$ with integer entries defines a morphism (homomorphism) $\phi(M)$ on $\T$ defined by:
\begin{equation}
\phi(M) : (x_1 , \ldots , x_n) \mapsto \left ( \prod_l x_l^{M_{1l}}, \ldots , \prod_l x_l^{M_{nl}} \right  ).
\end{equation}

We have: $\phi(M\cdot N)=\phi(M)\circ\phi(N)$. Also, $\phi(M)$ is dominant if and only if $M$ is non-singular (i.e. has non-zero determinant); in which case $\phi(M)$ is \'etale of degree $\det(M)$. A monomial correspondence is any correspondence on the torus $\T$ that given by $(\T, \phi(M), \phi(N))$ where $M , N$ are matrices with integer entries, and $M$ is non-singular. The correspondence $(\T, \phi(M), \phi(N))$ is dominant if and only if $N$ is non-singular as well.

\begin{ex}\label{ex:iteratesquare}
Let $n=1$, so $\T$ is the one-dimensional torus $\KK^*$. Let $M=N=[2],$ so both $\phi(M)$ and $\phi(N)$ are the squaring map $z\mapsto z^2$. Here, we compute explicitly the second iterate $g^2$ of the correspondence $g=(\T, \phi(M), \phi(N))$. 
For clarity, we use different coordinates on each copy of $\T$ --- we include the coordinate as a subscript. We have the following commutative diagram:
\begin{equation*}
\xymatrix{  & & \T_{\phi(N)}\times_{\phi(M)}\T \ar[rd]^{\pi_1} \ar[ld]_{\pi_2} &  & &  \\
 & \T_x \ar[rd]_{\substack{\phi(N)\\t=x^2}} \ar[ld]_{\substack{\phi(M)\\s=x^2}}  & & T_y  \ar[rd]^{\substack{\phi(N)\\u=y^2}} \ar[ld]_{\substack{\phi(M)\\t=y^2}} & \\
  \T_s & & \T_t & & \T_u}
\end{equation*}
\end{ex}
The fibered product $\T_{\phi(N)}\times_{\phi(M)}\T$ is, explicitly the subset $\{(x,y)\in\T^2\quad|\quad x^2=y^2\}$, i.e. the union of two lines $L_1=\{(x,y)\in\T^2\quad|\quad x=y\}$ and $L_2= \{(x,y)\in\T^2\quad|\quad x=-y\}$. Thus the second iterate $g^2$ is the correspondence: $((L_1)_x\sqcup (L_2)_x, x\mapsto x^2, x\mapsto x^2)$, which can be thought of as a `union' of two monomial correspondences but isn't one itself. Observe that the graph of $g^2$ is reducible; in contrast, the graph of a monomial correspondence is required by the definition to be a torus, thus irreducible. 
\section{Dynamics of monomial correspondences}

While the composition of monomial maps $\phi(M)$ and $\phi(N)$ is itself a monomial map, namely $\phi(M\cdot N)$, the composition of two monomial correspondences may not be monomial. Also, for a fixed monomial correspondence $(\T, \phi(M), \phi(N))$, its iterates may not be monomial, see Example \ref{ex:iteratesquare} above. The essential ingredient to understand the iterates of monomial correspondences is the following lemma, which relates the dynamical behavior of the monomial correspondence $(\T, \phi(M), \phi(N))$ to the dynamical  behaviour of $\phi(N\cdot \Com(M))$, where $\Com(M)$ is the transpose of the cofactor matrix of $M$. 

\begin{lem} \label{lem_crucial_monomial} Fix two $n \times n$ integer matrices $M$ and $N$, both with non-zero determinant and take $f$ to be the monomial correspondence $f := (\T, \phi(M), \phi(N))$ and  $P = N \cdot \Com(M)$ where $ \Com(M)$ is the transpose of the cofactor matrix of $M$.  
For any integer $p \geqslant n$, the following diagram is commutative:
\begin{equation*}
\xymatrix{ & \Gamma_p \ar[rd]^{v_p} \ar[ld]_{u_p} & \\
 \T \ar[rrd]^{\phi(P^p)} & & \T \ar[d]^{\phi(\det(M)^p \Id)} \\
  & & \T}
\end{equation*}
where  $f^p := (\Gamma_p, u_p, v_p)$ denotes the $p$-th iterate of $f$. 
\end{lem}
\begin{proof}
We induct on $p$. Throughout, we use the facts that $\phi(A\cdot B)=\phi(A)\circ\phi(B)$, and that a scalar matrix (in particular $\det(M)\Id$) commutes with all matrices, so $\phi(det(M)\Id$ commutes with all monomial maps. 

\noindent \textbf{Base Case:} $p=1$. Then $(\Gamma_p, u_p, v_p)=(\T, \phi(M), \phi(N))$. So 
\begin{align*}
\phi(P^p)\circ u_p&=\phi(P)\circ \phi(M)\\&=\phi(N\cdot \Com(M))\circ\phi(M)\\&= \phi(N)\circ\phi(\Com(M)))\circ\phi(M)\\&=\phi(\Com(M)))\circ\phi(M)\circ\phi(N)\\&=\phi(\Com(M)\cdot M)\circ\phi(N)\\&=\phi(det(M)\Id)\circ\phi(N)\\&=\phi(\det(M^p)\Id)\circ v_p.
\end{align*}

\noindent\textbf{Inductive Hypothesis:} For some $p>0$, $\phi(P^p)\circ u_p=\phi(\det(M^p)\Id)\circ v_p$.

\noindent\textbf{Inductive Step:} We have the following commutative digram:
\begin{equation*}
\xymatrix{  & & \Gamma_{p+1} \ar[rd]^{y} \ar[ld]_{x} &  & &  \\
 & \T \ar[rd]_{\phi(N)} \ar[ld]_{\phi(M)}  & & \Gamma_p  \ar[rd]^{v_p} \ar[ld]_{u_p} & \\
  \T & & \T & & \T}
\end{equation*}
Here, the square is Cartesian, $u_{p+1}=\phi(M)\circ x$, and $v_{p+1}= v_p\circ y$.

Now, we post-compose $v_{p+1}$ by $\phi(\det(M)^{p+1} \Id)$ to obtain:
\begin{align*}
\phi(\det(M^{p+1})\Id)\circ v_{p+1} &= \phi(\det(M)\Id)\circ \phi(\det(M)^{p}\Id)\circ  v_p\circ y\\
&= \phi(\det(M)\Id)\circ \phi(P^p)\circ u_p \circ y\\
&= \phi(\det(M)\Id)\circ \phi(P^p)\circ \phi(N) \circ x\\
&=  \phi(P^p)\circ \phi(N) \circ\phi(\det(M)\Id) \circ x\\
&= \phi(P^p)\circ \phi(N) \circ\phi(\Com(M)\cdot M) \circ x\\
&=  \phi(P^p)\circ \phi(N)\circ \phi(\Com(M))\circ\phi(M) \circ x\\
&=  \phi(P^p)\circ \phi(P)\circ\phi(M) \circ x\\
&=  \phi(P^{p+1})\circ\phi(M) \circ x\\
&=  \phi(P^{p+1})\circ u_{p+1}
\end{align*}
Here, the second equality follows from the inductive hypothesis and the third equality follows from the commutativity of the Cartesian square. 
\end{proof}


\subsection{Dynamical degree of monomial correspondences}

\begin{thm} \label{thm_action_numerical}  Fix two $n\times n$ integer matrices $M$ and $N$, both with non-zero determinant, and take $f$ to be the monomial correspondence $f := (\T, \phi(M), \phi(N))$. Denote by $P = N \cdot  \Com(M)$ where $ \Com(M)$ is the transpose of the cofactor matrix of $M$.  
For any toric compactification $X$ of $\T$, for any integer $p\geqslant 1$ and for any integer $l \leqslant n$, the following equality holds on $\num^l(X)\otimes \Q$:
\begin{equation*}
(f^p)^\bullet = \dfrac{1}{|\det(M)|^{lp - p}}\phi(P^p)^\bullet 
\end{equation*}
\end{thm}

\begin{proof}
Fix a toric compactification $X$ of $\T$. We observe that for any scalar matrix $a \Id$, the map $\phi(a\Id)$ is a regular morphism on $X$ and that the pullback satisfies:
\begin{equation*}
\phi(a\Id)^* = |a|^{l} \Id
\end{equation*}
on $\num^l(X)$.
Thus the morphism $\phi(\det(M)^p\Id)$ on $\T$ induces a regular morphism on $X$ and the pullback satisfies:
\begin{equation*}
\phi(\det(M)^p\Id)^* = |\det(M)|^{pl} \Id
\end{equation*}
on $\num^l(X)$.
As a result, we have the following equality on $\num^l(X)\otimes \Q$ 
\begin{equation} \label{eq_division_trick}
\dfrac{1}{|\det(M)|^{pl}} \phi(\det(M)^p\Id)^* = \Id.
\end{equation} 
By Lemma \ref{lem_crucial_monomial}, for any integer $p$, the following diagram is commutative:
\begin{equation*}
\xymatrix{ & \Gamma_p \ar[rd]^{v_p} \ar[ld]_{u_p} & \\
 \T \ar[rrd]^{\phi(P^p)} & & \T \ar[d]^{\phi(\det(M)^p \Id)} \\
  & & \T}
\end{equation*}
where $P = N \cdot \Com(M)$. Note that $u_p$ and $v_p$ are \'etale of degrees $\abs{\det(M)}^p$ and $\abs{\det(N)}^p$ respectively. First, we choose a birational modification $\pi:\tilde{X}\to X$ so that the map $\tilde{\phi}(P^p):\tilde{X}\to X$ induced by $\phi(P^p)$ is regular. Then, we choose a compactification $\bar{\Gamma_p}$ of $\Gamma_p$ so that the map induced by $u_p$ and $v_p$ from $\bar{\Gamma_p}$ to $X$ are regular, and the map $\tilde{u}_p:\bar{\Gamma_p}\to\tilde{X}$ is also regular. We thus obtain the following commutative diagram:
\begin{equation*}
\xymatrix{ & \bar{\Gamma_p} \ar[d]^{\tilde{u}_p}\ar[rrd]^{v_p} \ar[ld]_{u_p} &  \\
 X \ar@{-->}[rrrd]_{\phi(P^p)} & \tilde{X} \ar[rrd]^{\tilde{\phi}(P^p)} \ar[l]_{\pi} & &  X \ar[d]^{\phi(\det(M)^p \Id)} \\
  & & & X}
\end{equation*}
Note that $\tilde{u}_p$ as a map from $\bar{\Gamma_p}$ to $\tilde{X}$ is generically finite of degree $\abs{\det(M)}^p$. Thus, on  $\num^l(\tilde{X})\otimes \Q$, we have that $({\tilde{u}_p})_* \circ \tilde{u}_p^* = |\det(M)|^{p} \Id$. By the above diagram and \eqref{eq_division_trick}, we compute the pullback $(f^p)^\bullet$ on $\num^l(X)\otimes \Q$:
\begin{align*}
(f^p)^\bullet &= {u_p}_* \circ v_p^*\\
 &= \dfrac{1}{|\det(M)^{pl}|}  {u_p}_* \circ v_p^* \circ  \phi(\det(M)^p\Id)^* \\
 &= \dfrac{1}{|\det(M)^{pl}|}  {u_p}_* \circ\tilde{u}_p^* \circ  \tilde{\phi}(P^p)^* \\
 &= \dfrac{1}{|\det(M)^{pl}|}  \pi_*\circ({\tilde{u}_p})_* \circ\tilde{u}_p^* \circ  \tilde{\phi}(P^p)^* \\
&= \dfrac{\abs{\det(M)}^p}{|\det(M)^{pl}|}  \pi_* \circ  \tilde{\phi}(P^p)^* \\
 & =  \dfrac{1}{|\det(M)|^{lp-p}}    \phi(P^p)^* \\
\end{align*}
where the second-to-last equality follows from the fact that $({\tilde{u}_p})_* \circ \tilde{u}_p^* = |\det(M)|^{p} \Id$.

\end{proof}

From the behavior of the corresponding monomial map associated to the correspondence on the numerical groups, we deduce Theorem A.

\subsection{Proof of Theorem A}
%
%
Let $X$ be any toric compactification of $\T$ and $P = N \cdot \Com(M)$. For any integer $p>0$, Theorem \ref{thm_action_numerical} asserts that:
\begin{equation*}
(f^p)^\bullet =  \dfrac{1}{|\det(M)|^{lp-p}} \phi(P^p)^*,
\end{equation*}
on $\num^l(X)\otimes \Q$.
We fix a norm $|| \cdot ||$ on $\num^l(X)\otimes \Q$ and compute the $l$-dynamical degree of the correspondence using the formula:
\begin{equation*}
\lambda_l(f) = \lim_{p\rightarrow +\infty}|| (f^p)^\bullet ||^{1/p}.
\end{equation*}
The previous expression thus gives:
\begin{align*}
\lambda_l(f) &= \dfrac{1}{|\det(M)|^{l-1}} \lim_{p\rightarrow +\infty} || \phi(P^p)^*||^{1/p}\\
&=\dfrac{1}{|\det(M)|^{l-1}} \lambda_l(\phi(P))
\end{align*} 
Let $\rho_1', \ldots , \rho_l'$ be the $l$ largest eigenvalues (in absolute value, and counted with multiplicity) of the matrix $P=N\cdot \Com(M)=\det(M)\cdot N\cdot M^{-1}$. Let $\rho_1, \ldots , \rho_l$ be the $l$ largest eigenvalues of the matrix $N\cdot M^{-1}$. Then $\rho_i'=\det(M)\cdot \rho_i$. By \cite{favre_wulcan, lin} applied to the monomial map associated to $P$, we have:
\begin{align*}
\lambda_l(f) &= \dfrac{1}{|\det(M)|^{l-1}} |\rho_1'| \cdot \ldots \cdot |\rho_l'|\\
&=\dfrac{1}{|\det(M)|^{l-1}}\abs{\det(M)} |\rho_1| \cdot \ldots \cdot \abs{\det(M)}|\rho_l|\\
&= \dfrac{\abs{\det(M)}^l}{|\det(M)|^{l-1}} |\rho_1| \cdot \ldots \cdot |\rho_l|\\
&=|\det(M)| |\rho_1| \cdot \ldots \cdot |\rho_l|
\end{align*}


\begin{rem}
It follows immediately from the projection formula that for a general correspondence $(\Gamma,\pi_1,\pi_2)$ on an $n$-dimensional variety, the $l$-th dynamical degree of $(\Gamma,\pi_1,\pi_2)$ is equal to the $(n-l)$-th dynamical degree of $(\Gamma,\pi_2,\pi_1)$. The following computation provides a sanity-check for Theorem A: Let $\rho_1,\ldots,\rho_n$ be the eigenvalues of $N\cdot M^{-1}$, arranged so that the sequence of absolute values is non-increasing. Then $\frac{1}{\rho_n},\ldots, \frac{1}{\rho_1}$ are the eigenvalues of $M\cdot N^{-1}$, again arranged so that the sequence of absolute values is non-increasing. We note that:
\begin{align*}
&\text{$l$-th dynamical degree of $(\T,\phi(M),\phi(N))$}\\&= |\det(M)| |\rho_1| \cdot \ldots \cdot |\rho_l|\\&=\abs{\frac{\det(M)\cdot \det(N\cdot M^{-1})}{|\rho_{l+1}| \cdot \ldots \cdot |\rho_n|}}\\
&=\abs{\det(N)}\cdot\frac{1}{|\rho_{n}|} \cdot \ldots \cdot \frac{1}{|\rho_{l+1}|}\\
&=\text{$(n-l)$-th dynamical degree of $(\T,\phi(N),\phi(M))$}
\end{align*}
\end{rem}

\subsection{Proof of Theorem B}

Fix a monomial correspondence $f := (\T, \phi(M), \phi(N))$ for $M,N \in \GL_n(\ZZ)$ whose $l$-dynamical degree satisfies the condition:
\begin{equation*}
\lambda_l^2(f) > \lambda_{l+1}(f) \lambda_{l-1}(f).
\end{equation*}
Using Theorem A, this condition implies that:
\begin{equation*}
\dfrac{\lambda_l(f)^2}{\lambda_{l+1}(f) \lambda_{l-1}(f)} = \dfrac{\lambda_l(\phi(P))^2}{\lambda_{l+1}(\phi(P)) \lambda_{l-1}(\phi(P))}>1, 
\end{equation*}
where $P$ is the matrix $N \cdot \Com(M)$. This prove that $\phi(P)$ satisfies the conditions of \cite[Theorem D]{favre_wulcan}, hence  the asymptotic growth of the $l$-degree of $\phi(P)$ is given by:
\begin{equation*}
\deg_l(\phi(P)^p) = C \lambda_l(\phi(P))^p + O \left (p^r \left (\dfrac{\lambda_{l-1}(\phi(P)) \lambda_{l+1}(\phi(P))}{\lambda_l(\phi(P))} \right )^p \right ),
\end{equation*}   
as $p\rightarrow +\infty$ where $r$ is an integer.

Using Theorem \ref{thm_action_numerical} and the fact that $\lambda_l(f) |\det(M)|^{l-1}= \lambda_l(\phi(P))$, we deduce that:
\begin{equation*}
\deg_l(f^p) = C \lambda_l(f)^p + O \left ( p^r\left ( \dfrac{\lambda_l(\phi(P))^2}{\lambda_{l+1}(\phi(P)) \lambda_{l-1}(\phi(P))} \right )^p \right ), 
\end{equation*}
and the theorem is proved.


\subsection{Proof of Theorem C}

Fix  $f = (\T, \phi(M), \phi(N))$ a monomial correspondence where $M,N$ are two matrices in $\GL_n(\Z)$ such that the eigenvalues of the matrix $N M^{-1}$ are all real, distinct and positive. This implies that the eigenvalues of the matrix $P= N\cdot \Com(M)$ also satisfy the assumption of the theorem. 
By \cite[Theorem A]{lin_wulcan}, we can find a toric compactification $X$ of the torus $\T$, with at worst quotient singularities, on which the monomial map $\phi(P)$ is $k$-stable. Using Theorem \ref{thm_action_numerical} we deduce that that the $k$-stability of $\phi(P)$ on $X$ implies the $k$-stability of the correspondence $f$ on $X$, i.e. we have on $\num^l(X)\otimes \Q$:
\begin{align*}
(f^p)^\bullet &=\frac{1}{\abs{\det(M)}^{p(l-1)}}\phi(P^p)^*\\
&=\left(\frac{1}{\abs{\det(M)}^{(l-1)}})^p\phi(P^*\right)^p\\
&=\left(\frac{1}{\abs{\det(M)}^{(l-1)}}\phi(P^*)\right)^p=(f^\bullet)^p
\end{align*}

\subsection{Proof of Theorem D}

We fix an embedding of $\T$ into $\PP^n$. 
Fix also an integer $p \geqslant 1$ and denote by $(\Gamma_p,u_p,v_p)$ the  $p$-th iterate of the correspondence $f$.
By Lemma \ref{lem_crucial_monomial}, we can find a toric compactification $\bar \Gamma_p$ of $\Gamma_p$ such that the following diagram is commutative:
\begin{equation*}
\xymatrix{ & \bar\Gamma_p \ar[rd]^{v_p} \ar[ld]_{u_p} & \\
 \PP^n \ar@{-->}[rrd]^{\phi(P^p)} & & \PP^n \ar[d]^{\phi(\det(M)^p \Id)} \\
  & & \PP^n}
  \end{equation*}
Fix $D$ an ample divisor in $\PP^n$, a point $x \in \T(\QQ)$ and denote by $h_D$ the Weil-height associated to $D$. 
Observe that since $\phi(\det(M)^p \Id)$ is the map that raises each homogenous coordinate of $\P^n$ to its $\det(M)^p$-th power, we have:
\begin{equation*}
h_D(\phi(\det(M)^p \Id) x) = |\det(M)|^p h_D(x)
\end{equation*}
for any point $x \in \T(\bar \QQ)$. 
As a result, we compute the height of the image of the point $x\in \T(\QQ)$ as:
\begin{equation*}
h_D({v_p}_*u_p^*[x]) = \dfrac{1}{|\det(M)|^p} h_D(\phi(\det(M)^p \Id)_*{v_p}_* u_p^* [x]).
\end{equation*}
As the above diagram is commutative, we thus obtain:
\begin{equation*}
h_D({v_p}_* u_p^* [x]) = \dfrac{1}{|\det(M)|^p} h_D(\phi(P^p )_*{u_p}_* u_p^* [x]).
\end{equation*}
As the cycle ${u_p}_* u_p^* [x]$ equals $|\det(M)|^p [x]$, we deduce that:
\begin{equation*}
h_D(f^p(x)) = h_D(\phi(P^p)(x)).
\end{equation*}
By \cite[Theorem 4]{silverman}, the asymptotic limit  
\begin{equation*}
\limsup h_D(\phi(P^p)(x))^{1/p} 
\end{equation*}
is well-defined and belongs to the set $\{1 , |\rho_1|, |\rho_2| , \ldots , |\rho_n| \}$ where each $\rho_i$ is an eigenvalue of $P$ and the result is proved. 

\begin{rem} Observe that  the Zariski density for the orbit of the monomial correspondence $f = (\T, \phi(M), \phi(N))$ does not imply the Zariski density of the  orbit of $\phi(P)$ where $P = N \cdot \Com(M)$. One can take  for example the case where $M = 2 \Id$, $N=\Id$ and $\T = \C^*$. The correspondence $f=(\T,\phi(M),\phi(N))$ is the ``square root" correspondence, i.e. $f(z)=\{+\sqrt(z),-\sqrt(z)\}$. Consider the point $z_0=1\in\T$: its orbit under $f$ is the set $\{z \in \bar{\Q}\quad|\quad z^{(2^k)}=1\}$, thus Zariski dense. However, under the monomial map $\phi(P) = z\mapsto z^2$, the orbit of $z_0$ is $\{z_0\}$, not Zariski dense.  
As a result, one cannot directly apply Silverman's result to prove that  a point $x\in \T(\bar \Q)$ whose orbit for $f$ is Zariski dense satisfies $\alpha_f(x)=\lambda_1(f)$. This later statement is often refered to as Kawaguchi-Silverman's conjecture (\cite{kawaguchi_silverman_arithmetic})
\end{rem}

\bibliographystyle{amsalpha}
\bibliography{HurwitzRefs}
\end{document}